\documentclass[12pt]{amsart}
\usepackage[colorlinks=true,citecolor=magenta,linkcolor=magenta]{hyperref}
\usepackage{amssymb,verbatim,amscd,amsmath,graphicx}
\usepackage{enumerate}
\usepackage{graphicx}
\usepackage{caption}
\usepackage{cleveref}
\usepackage{subcaption}
\usepackage{pdfsync}
\usepackage{color,colortbl}
\usepackage{fullpage}
\usepackage{bbm}
\usepackage{comment}
\usepackage{tikz}
\numberwithin{equation}{section}

\newtheorem{theorem}{Theorem}[section]
\newtheorem{lemma}[theorem]{Lemma}
\newtheorem{proposition}[theorem]{Proposition}
\newtheorem{corollary}[theorem]{Corollary}

\theoremstyle{definition}
\newtheorem{definition}[theorem]{Definition}

\newtheorem{def-prop}[theorem]{Definition-Proposition}

\newtheorem{remark}[theorem]{Remark}
\newtheorem{example}[theorem]{Example}

\DeclareMathOperator{\Ass}{Ass}

\DeclareMathOperator{\lcm}{lcm}

\DeclareMathOperator{\svd}{svd}

\newcommand{\PP}{{\mathbb P}}
\newcommand{\ZZ}{{\mathbb Z}}
\newcommand{\NN}{{\mathbb N}}

\newcommand{\kk}{{\mathbbm k}}

\def\I{{\mathcal I}}

\def\M{{\mathcal M}}

\def\R{{\mathcal R}}

\def\R{{\mathcal R}}

\def\pp{{\frak p}}

\def\1{{\bf 1}}
\def\0{{\bf 0}}

\title{Standard Veronese degree and star configurations}

\author{Th\'ai Th\`anh Nguy$\tilde{\text{\^E}}$n}
\address{University of Dayton, Department of Mathematics,
	300 College Park, Dayton, Ohio, USA \\
	and University of Education, Hue University, 34 Le Loi, Hue, Vietnam}
\email{tnguyen5@udayton.edu}

\begin{document}

\begin{abstract}
We show that the standard Veronese degree of the defining ideal $I_{n,h}$ of a monomial star configuration is $\lcm(1,2,\ldots,h)$, proving a conjecture by Grifo and Seceleanu. We also provide a counter-example for their question regarding upper bounds for the standard Veronese degree of monomial ideals.
\end{abstract}

\maketitle

\section{Introduction}

Over recent years, \textit{star configurations} in $\PP^n$ and their generalized versions have emerged as an interesting family of projective varieties that serve as examples and counterexamples for various problems in commutative algebra and algebraic geometry. Generally speaking, a star configuration (generalized star configuration) is a union of complete intersection linear spaces (complete intersection of hypersurfaces, respectively). These configurations and their defining ideals have been shown to exhibit many interesting extremal properties. For example, star configurations form important examples in relation to the ideal containment problem and resurgence numbers \cite{BoH, bghn2022demailly}. Many aspects of (generalized) star configurations have been studied extensively, including the minimal generators and minimal graded free resolutions of their ideals \cite{GHM2013,PS2015, Galetto20} and of the \textit{symbolic powers} of their ideals \cite{Mantero, LinShen}. Some of many other papers that have contributed to our understanding of star configurations include \cite{WaldschmidtSpecialConfig, CGVT2015, AGT2017}, as well as \cite{Toh2017}, which explores an interesting connection to subspace arrangements. 

Symbolic powers of an ideal is one of the central objects in commutative algebra and algebraic geometry. Let $R$ be a commutative ring and let $I \subseteq R$ be an ideal. For $n \in \NN$, the $n$-th symbolic power of $I$ is defined to be
 	$$I^{(n)} = \bigcap_{\pp \in \Ass(I)} \left(I^nR_\pp \cap R\right).$$

In this paper, we are interested in finding the \textit{standard Veronese degree} of the ideal of certain star configurations. In a recent survey \cite{symbolicReesSurvey} on symbolic Rees algebras, Grifo and Seceleanu defined the standard Veronese degree of an ideal $I$ to be 
$$\svd(I):= \inf \{ d \mid I^{(dk)}=(I^{(d)})^k \text{  for all  } k\ge 1 \}.$$

In other words, $\svd(I)$ is the least integer $c$ such that the $c$-th Veronese subalgebra of the symbolic Rees algebra of $I$ (see next section for the definition) is standard graded. This is a special instance of a more general problem: If $A=\bigoplus_{i=0}^\infty A_i$ is a finitely generated, graded $R$-algebra over a ring $R=A_0$, then what is the least integer $c\ge 1$, denoted $\svd(A)$, such that the $c$-th Veronese subalgebra $A^{(c)}=\bigoplus_{i=0}^\infty A_{ic}$ of $A$ is standard graded? Driven by a recent, active trend in studying nonstandard graded algebras, studying the standard Veronese degree of graded algebras provides a bridge back to the standard graded setting where extensive literature has already been established. Geometrically, finding $\svd(A)$ yields an effective bound which guarantees that $\text{Proj}(A^{(c)})$ admits a very ample embedding into a projective space. When $A$ is the Rees algebra of a graded family of ideals, understanding $\svd(A)$ provide useful tools for studying the asymptotic behavior and invariants of the given graded family.

If the maximum degree of minimal generators of $A$ is $\omega$, a general upper bound for $\svd(A)$ is known to be $\omega \cdot \omega!$, see, for instance, \cite{symbolicReesSurvey}. A general upper bound for $\svd(I)$ can be found in \cite{Rees58} which states that if the symbolic Rees algebra of $I$ is generated in degrees $a_1,a_2,\ldots,a_s$ then $\svd(I)\le s\lcm(a_1,\ldots ,a_s)$. These bounds are far from sharp in general. Specific upper bounds were studied including for Fermat point configurations in \cite{NagelSeceleanu} and for certain space monomial curve families in \cite{Mor91}. 

Our approach to finding the standard Veronese degree is to reduce it to a problem on minimal nonnegative solutions of linear Diophantine equations. This simple yet effective approach reveals connections between our problem and minimal solutions of linear Diophantine equations, zero-sum sequences, and the search for Hilbert bases of appropriate convex polyhedra \cite{HenkWeis, Sissokho21}. Our main result states that if $A$ is a graded algebra over $A_0$ generated in degrees at most $h$, then the $d$-th Veronese subalgebra $A^{(d)}$ is standard graded, where $d = \lcm(1,\dots,h)$. In particular, $\svd(A) \le \lcm(1,\dots,h)$, see \Cref{cor.SVDAlgebra}. Applying this to the ideal $I$ of a \textit{codimension $h$ monomial star configuration} (see next section for the definition), we obtain $\svd(I) = \lcm(1,\ldots ,h)$, see \Cref{thm.Star}. This proves a conjecture in \cite{symbolicReesSurvey}. In \cite[Question 3.10]{symbolicReesSurvey}, it was also asked if 
\[
\svd(I) \le \text{ the lcm of the degrees of any set of algebra generators for } \R_s(I)
\]
for any monomial ideal. We provide examples to answer this question negatively.

\par
\vspace{1em}
\noindent
{\bf Acknowledgments.}
The author thanks an anonymous referee for a careful read and constructive comments. After submitting this paper, the author learned in a personal communication with P. Mantero and V. Nguyen that they had independently proved the same results using a different method in their ongoing project. The author acknowledges support from the AMS-Simons Travel Grant.


\section{Preliminaries}

Let $R$ be a standard graded algebra over a field. A collection $\{I_n\}_{n \ge 0}$ of ideals in $R$ is called a \emph{graded family} of ideals if $I_p \cdot I_q \subseteq I_{p+q}$ for all $p, q \ge 0$. By convention, we will assume that $I_0 = R$. A graded family $\{I_n\}_{n \ge 0}$ is called a \emph{filtration} if, in addition, $I_p \supseteq I_{p+1}$ for all $p \ge 0$.
The \emph{Rees algebra} of a graded family $\I = \{I_n\}_{n \ge 0}$ of ideals in $R$ is defined by
		$$\R(\I) = \bigoplus_{n \ge 0} I_nt^n \subseteq R[t].$$
Note that $\R(\I)$ is a graded algebra. A graded family $\I$ of ideals in $R$ is said to be \emph{Noetherian} if its Rees algebra is a Noetherian ring. Important examples of graded families include the filtration of all ordinary powers $(I^n)_{n\ge 0}$, the integral closure of powers $(\overline{I^n})_{n\ge 0}$, and the symbolic powers $(I^{(n)})_{n\ge 0}$, of an ideal $I$. We shall use $\R_s(I)$ to denote the Rees algebra of the symbolic power filtration of $I$.

For a graded algebra $A=\bigoplus_{i=0}^\infty A_i$ over $A_0=R$ and an integer $c \ge 1$, the $c$-th Veronese subring  of $A$ is 
$A^{(c)}=\bigoplus_{i=0}^\infty A_{ic}$.
Note that we can view $A^{(c)}$ as a graded algebra over $R$, by setting $\deg(A_{ic})=i$.
It is well-known that $A$ is a finitely generated $R$-algebra if and only if there is an integer $c \ge 1$ such that $A^{(c)}$ is a standard graded $R$-algebra, see, for instance, \cite[Theorem 2.1]{HHT2007}. In particular, a graded family $\I$ of ideals in $R$ is Noetherian if and only if there exists $c \ge 1$ such that $I_{ck}=I_c^k$ for all $k\ge 1$. For a similar statement in terms of the integral closure of $I_n$'s, see \cite[Theorem 3.4]{TaiThaiNOBody}.\par

\vspace{0.5em}

Let $R = \kk[x_1, \dots, x_n]$ and let $1\le h < n$. The defining ideal of the \textit{codimension $h$ monomial star configuration} is given by
	$$I_{n,h} = \bigcap_{1 \le i_1 < \dots < i_h \le n} (x_{i_1}, \dots, x_{i_h}).$$

If, in the intersection above, we replace the variables $x_i$'s by homogeneous polynomials $f_i$'s, and assume that any collection of at most $h$ distinct elements $f_i$'s forms a complete intersection, we then get the defining ideal of a so-called generalized star configuration that has been studied extensively in, for instance, \cite{Mantero, LinShen, GHMN17}. In particular, it is known that if $I$ is the defining ideal of a generalized star configuration then $\R_s(I)$ is finitely generated, hence, there exists $c\ge 1$ such that $I^{(ck)}=(I^{(c)})^k$ for all $k$, thus, $\svd(I)$ is finite.\par

\vspace{0.5em}

Let $k_1,k_2,\ldots, k_s$ be integers and consider the linear Diophantine equation $k_1x_1+k_2x_2+\ldots +k_sx_s=0$. We say that a solution $(a_1,a_2,\ldots,a_s) \in \ZZ_{\ge 0}^s$ is a \textit{minimal nonnegative solution} of the equation if we cannot write it as a sum of two other nonnegative solutions 
\[
(a_1,a_2,\ldots,a_s) = (a_1',a_2',\ldots,a_s') + (a_1'',a_2'',\ldots,a_s''),
\]
where $0\le a_i', a_i'' \le a_i$ for all $i=1,2,\ldots,s$. For the rest of the paper, we will refer to minimal nonnegative solutions as \textit{minimal solutions}.

\begin{definition}
    \label{def.condM}
    Let $k_1,k_2,\ldots,k_s,d \in \ZZ_{>0}$. We say that the equation (in variables $x_1,\ldots,x_s, y$) $k_1x_1+k_2x_2+\ldots +k_sx_s = dy$ \textit{satisfies condition $\mathcal{M}$} if all minimal solutions have $y=1$.
\end{definition}

The following simple observation allows us to study the standard Veronese degree via linear Diophantine equations that satisfy condition $\mathcal{M}$.

\begin{proposition}
    \label{prop.SVDandMinSol}
    Let $A$ be a graded algebra over $A_0=R$, generated in degrees $a_1,\ldots, a_s >0$. Let $d$ be a positive integer such that the equation $a_1x_1+a_2x_2+\ldots +a_sx_s = d$ has nonnegative integral solutions. If the equation $a_1x_1+a_2x_2+\ldots +a_sx_s = dy$ satisfies condition $\mathcal{M}$, then the $d$-th Veronese subalgebra $A^{(d)}=\bigoplus_{i=0}^\infty A_{id}$ is standard graded.
\end{proposition}

\begin{proof}
    Because $A=R[A_{a_1},A_{a_2},\ldots, A_{a_s}]$, for any $k\ge 2$, we have: 
$$A_{dk}=\sum_{a_1n_1+a_2n_2+\ldots + a_sn_s=dk}A_{a_1}^{n_1}A_{a_2}^{n_2}\ldots A_{a_s}^{n_s}.$$

Since $a_1x_1+a_2x_2+\ldots +a_sx_s = dy$ satisfies condition $\mathcal{M}$, $(n_1,n_2,\ldots,n_s)$ can be written as a sum of $k$ nonnegative vectors of the form $(n_1',n_2',\ldots ,n_s')$ where $a_1n_1'+a_2n_2'+\ldots + a_sn_s'=d$. This shows that 

\[
A_{dk} \subseteq \left( \sum_{a_1n_1' + \ldots + a_sn_s' = d} A_{a_1}^{n_1'}A_{a_2}^{n_2'}\cdots A_{a_s}^{n_s'}  \right)^k = (A_d)^k \subseteq A_{dk}.
\]

Therefore, $A_{dk}=(A_d)^k$ for all $k\ge 1$, or equivalently, $A^{(d)}$ is standard graded. 
\end{proof}

\begin{remark}
    \label{rem.SVDandMinSol}
    The condition $\mathcal{M}$ only gives a sufficient condition for Veronese subalgebras to be standard graded. Consider the graded algebra $A = \kk[t^2, t^3]$ generated in degrees $2$ and $3$. While $A^{(5)}=\kk[t^5]$ is standard graded, the equation $2x_1 + 3x_2 = 5y$ fails to satisfy condition $\mathcal{M}$. In fact, the solution $(x_1, x_2, y) = (5, 0, 2)$ is minimal. 
\end{remark}


\section{Standard Veronese Degree of Star Configurations}

In \cite[Example 6.7]{TaiThaiNOBody}, it is shown that $d_h:=\lcm(1,2,\ldots ,h)$ is the least integer $c$ such that $I_{n,h}^{(ck)}=\overline{(I_{n,h}^{(c)})^k}$ for every $k\ge 1$. In \cite[Example 3.9]{symbolicReesSurvey}, it is shown that if $I_{n,h}^{(ck)}=(I_{n,h}^{(c)})^k$ for every $k\ge 1$ then $d_h \mid c$, or equivalently, $d_h \mid \svd(I_{n,h})$. In this section, we will show that, in fact, $\svd(I_{n,h})=d_h$.\par
\vspace{0.5em}

First, we prove the following simple observation about minimal nonnegative solutions of linear Diophantine equations that will be used later.

\begin{lemma}
    \label{lem.xi<n}
    Let $k_1,k_2,\ldots,k_s \in \ZZ_{>0}$ and $d=\lcm(k_1,\ldots ,k_s)$. Suppose that $(a_1,a_2,\ldots ,a_s, b)$ with $b\ge 2$ is a minimal solution of the equation $k_1x_1+k_2x_2+\ldots +k_sx_s = dy$ (in variables $x_1,\ldots,x_s,y$). Then $a_i\le \dfrac{d}{k_i}-1$ for all $i=1,\ldots ,s$. 
\end{lemma}

\begin{proof}
    Without loss of generality, suppose that $a_1 \ge \dfrac{d}{k_1}$. We can write
    \[
    (a_1,a_2,\ldots ,a_s, b) = (a_1-\dfrac{d}{k_1},a_2,\ldots ,a_s, b-1) + (\dfrac{d}{k_1},0,\ldots ,0, 1), 
    \]
    a sum of two nonnegative solutions. This shows that $(a_1,a_2,\ldots ,a_s, b)$ is not a minimal solution.
\end{proof}

Before proving the main result, we need the following numerical result.

\begin{lemma}
    \label{lem.combIneq}
    Let $d_n=\lcm(1,2,\ldots , n)$. Then $d_n >\dfrac{(n-2)(n+1)n}{4}$ for all $n\ge 1$.
\end{lemma}

\begin{proof}
    One can directly check that $d_n >\dfrac{(n-2)(n+1)n}{4}$ for all $1\le n \le 10$. Suppose that $n\ge 11$. Let $\psi(n)= \ln(d_n)$ be the second Chebyshev function of $n$. By \cite[Lemma 2]{Nagura}, $\psi(n)>0.916n-2.318$ for all $n\ge 1$. Thus, for $n\ge 11$, we get $$\psi(n)>0.916n-2.318>n\ln 2.$$
    On the other hand, by induction on $n$, we have $2^n \ge \dfrac{(n-2)(n+1)n}{4}$ for all $n\ge 11$. This concludes the proof.
\end{proof}

We will also use the following combinatorial result in the proof of our main result.

\begin{lemma}
    \label{pigeonhole}
    Let $a_1,a_2,\ldots ,a_n$ be a sequence of integers. Then for any $1\le m \le n$, there exists a subsequence $a_{i_1},a_{i_2},\ldots ,a_{i_k}$ with $k\ge n-m+1$ elements such that $m \mid (a_{i_1}+a_{i_2}+\ldots +a_{i_k})$.
\end{lemma}

\begin{proof}
    We use induction on $n\ge m$. First, consider the base case where $n=m$.
    Form $m$ partial sums of the sequence modulo $m$: $s_\ell = \displaystyle \sum_{i=1}^\ell a_i $ (mod $m$) for $1\le \ell \le m$. If any of these partial sums is $0$ then we are done. Otherwise, if those $s_\ell$'s all belong to $\{1,2,\ldots ,m-1 \}$, then by Pigeonhole principle, there exists $s_{j_1}, s_{j_2}$ with $j_2>j_1$ so that $s_{j_1}= s_{j_2}$. It follows that $m$ divides $\displaystyle \sum_{i=j_1+1}^{j_2} a_i$. The claim holds for the base case since $j_2-j_1\ge 1 = m-m+1$. Suppose that the claim holds for any sequence of $n\ge m$ integers. Consider a sequence $a_1,a_2,\ldots ,a_{n+1}$ of $n+1$ integers. Form $n+1$ partial sums $s_\ell = \displaystyle \sum_{i=1}^\ell a_i $ (mod $m$) and using the same argument as above, there exists $j_2>j_1$ such that $m$ divides $\displaystyle \sum_{i=j_1+1}^{j_2} a_i$. 

    If $j_2-j_1 \ge n-m+1$ then take $k=j_2-j_1$ and we are done. Otherwise, if $j_2-j_1 \le n-m$, then $m \le n-(j_2-j_1)$. Removing $a_{j_1+1},\ldots,a_{j_2}$ from the sequence and considering the remaining sequence of exactly $n-(j_2-j_1)$ elements, we can apply the induction hypothesis to obtain a subsequence of $k'$ elements $a_{i_1},a_{i_2},\ldots ,a_{i_{k'}}$ where $k'\ge n-(j_2-j_1)-m+1$ such that $m \mid (a_{i_1}+a_{i_2}+\ldots +a_{i_{k'}})$. Thus, we get a subsequence $\{a_{i_1},a_{i_2},\ldots ,a_{i_{k'}}, a_{j_1+1},\ldots,a_{j_2}\}$ of at least $n-m+1$ elements whose sum is divisible by $m$ as desired.
\end{proof}

\begin{proposition}
    \label{prop.minSol}
    For each $n\ge 1$, the equation $x_1+2x_2+\ldots +nx_n=d_ny$ where $d_n=\lcm(1,2,\ldots , n)$ satisfies condition $\mathcal{M}$.
\end{proposition}
\begin{proof}
    The case where $n=1$ is trivial. Assume that $n\ge 2$. Suppose that $(a_1,a_2,\ldots ,a_n,b)$ is a minimal solution of the equation $x_1+2x_2+\ldots +nx_n=d_ny$ with $b\ge 2$. By \Cref{lem.combIneq}, $d_nb \ge 2d_n > (n-2)(1+2+\ldots +n)$, thus, there exists $i$ such that $a_i\ge n-1$. Moreover, by \Cref{lem.xi<n}, we have $a_i \le \dfrac{d_n}{i}-1$. It follows that 
    \[
    a_1+2a_2 +\ldots + \widehat{a_i} + \ldots +na_n \ge 2d_n - ia_i \ge d_n +i,
    \]
    where $\widehat{a_i}$ indicates the absence of $a_i$ in the sum. Let $S = \sum_{j \ne i} j a_j \ge d_n + i$. We construct a finite, strictly decreasing sequence of integers $S_0, S_1, \ldots, S_N$ as follows. Set $S_0 = S$. To obtain $S_{k+1}$ from $S_k$, we choose any index $j \in \{1, \ldots, n\} \setminus \{i\}$ for which the current coefficient of $j$ is strictly positive, and we reduce that coefficient by $1$. This process terminates in $N = \sum_{j \ne i} a_j$ steps when all coefficients reach zero, yielding $S_N = 0$. Since the sequence $S_0, S_1, \ldots, S_N$ must cross the value $d_n + i$ exactly once, along this specific sequence of reductions, there is a unique index $m \ge 1$ where $S_m < d_n + i$ while $S_{m-1} \ge d_n + i$.

    Let $a_1', \dots, \widehat{a_i'}, \dots, a_n'$ be the specific coefficients that correspond to the sum $S_m$. By construction, $0 \le a_j' \le a_j$ for all $j \ne i$, and moreover $S_{m-1} - S_m \le n$. Thus, we obtain
    \[
    d_n+i-n \le a_1'+2a_2'+\ldots + na_n' < d_n + i.
    \]

    Note that there is no $a_i'$ in the above sum. Let $a=a_1'+a_2'+\ldots +a_n'$. If $a\le i-1$, then
    \[
    d_n+i-n \le a_1'+2a_2'+\ldots + na_n'\le n(a_1'+a_2'+\ldots +a_n')\le n(i-1).
    \]
    But this implies that
    \[
    d_n \le i(n-1) \le ia_i \le i\left(\dfrac{d_n}{i}-1 \right) = d_n-i,
    \]
    a contradiction. Therefore, $a\ge i$. Consider the following sequence of a numbers: the number $1$ appearing $a_1'$ times, the number $2$ appearing $a_2'$ times, $\ldots$, the number $n$ appearing $a_n'$ times, and the number $i$ not appearing at all. Since $a\ge i$, by \Cref{pigeonhole}, there exists a subsequence of $a_1''$ numbers $1$, $a_2''$ numbers $2$, $\ldots$, and $a_n''$ numbers $n$ with $0\le a_j''\le a_j'$ for all $1\le j \le n, j\ne i$, $a_1''+a_2''+\ldots+ a_n'' \ge a-i+1$ such that $i \mid (a_1''+2a_2''+ \ldots+ na_n'')$. Furthermore, since the total numbers that are not in the subsequence is $a-(a_1''+a_2''+ \ldots+ a_n'') \le i-1$, it follows that
    \[
    a_1''+2a_2''+ \ldots+ na_n'' \ge d_n+i-n - (i-1)n =  d_n-i(n-1). 
    \]
    Note that we also have $a_1''+2a_2''+ \ldots+ na_n'' < d_n+i$. Let $a_i''=\dfrac{d_n-(a_1''+2a_2''+ \ldots+ na_n'')}{i}$, then $a_i'' \in \ZZ$ because $d_n$ is the $\lcm(1,2,\ldots,n)$, and $0 \le a_i'' \le n-1$. Thus, we can write
    $$(a_1,\ldots ,a_n,b) = (a_1'',\ldots, a_i'', \ldots,a_n'',1)+(a_1-a_1'',\ldots, a_i-a_i'', \ldots, a_n-a_n'',b-1),$$
    where each of the right hand side is a nonnegative solution to the equation $x_1+2x_2+\ldots +nx_n=d_ny$. This contradicts the assumption that $(a_1,a_2,\ldots ,a_n,b)$ is a minimal solution, and hence, finishes the proof.    
\end{proof}

We deduce an immediate consequence of the above result.

\begin{corollary}
    \label{cor.SVDAlgebra}
    Let $A$ be a finitely generated, graded algebra over $A_0$ and $h$ be the maximum degree of a minimal algebra generator of $A$. Let $d_h= \lcm(1,\ldots,h)$. Then, the $d_h$-th Veronese subalgebra $A^{(d_h)}$ is standard graded. In particular, $\svd(A) \leq \lcm(1,\ldots,h)$.
\end{corollary}

\begin{proof}
    It is clear that the equation $x_1+2x_2+\ldots +hx_h=d_hy$ has nonnegative integral solutions. Since $A$ can be generated in degrees $1, 2, \ldots ,h$, by \Cref{prop.minSol} and \Cref{prop.SVDandMinSol}, $A^{(d_h)}$ is standard graded.
\end{proof}

We are now ready to prove the main result of this paper.

\begin{theorem}
    \label{thm.Star}
    For the ideal $I_{n,h}$ of the codimension $h$ monomial star configuration, we have $$\svd(I_{n,h})=\lcm(1,2,\ldots ,h).$$
\end{theorem}
\begin{proof}
    Denote $I=I_{n,h}$, then by \cite[Proposition 4.6]{HHT2007} or \cite[Example 3.9]{symbolicReesSurvey}, we have
    $$\R_s(I) = R[x_{i_1}x_{i_2}\ldots x_{i_{n-h+m}}t^m \mid 1\le m \le h, 1\le i_1 < i_2 < \ldots < i_{n-h+m} \le n].$$
    In particular, $\R_s(I)$ is generated in degrees $1,2,\ldots ,h$. By \Cref{cor.SVDAlgebra}, the $d_h$-th Veronese subalgebra $\R_s(I)^{(d_h)}$ is standard graded. On the other hand, if $c$ is an integer such that $\R_s(I)^{(c)}$ is standard graded, it is shown in \cite[Example 3.9]{symbolicReesSurvey} that $c$ is divisible by $d_h$, hence, $c\ge d_h$ (see also \cite[Example 6.7]{TaiThaiNOBody} for another proof that $c\ge d_h$). Therefore $\svd(I_{n,h})=\lcm(1,2,\ldots ,h).$
\end{proof}

\begin{corollary}
    \label{cor.genStar} For the ideal $I$ of a codimension $h$ generalized star configuration defined by forms of the same degree, we have $\svd(I)=\lcm(1,2,\ldots ,h)$. 
\end{corollary}

\begin{proof}
    This follows from \Cref{thm.Star} and \cite[Theorem 3.6]{GHMN17} noting that $\R_s(I)$ is still generated in degrees $1,2,\ldots,h$.
\end{proof}

It is natural to ask if the standard Veronese degree of a monomial ideal $I$ is bounded above by the least common multiple of the degrees of any set of algebra generators of $\R_s(I)$, \cite[Question 3.10]{symbolicReesSurvey}. We address this question in the next section.


\section{LCM of degrees of generators is not an upper bound for SVD}

In this section, we give a negative answer to \cite[Question 3.10]{symbolicReesSurvey}. Specifically, we provide an edge ideal $I$ such that 
\[
\svd(I) > \text{ the lcm of the degrees of a set of algebra generators for } \R_s(I).
\]

Notice first that for any ideal $I$, $\R_s(I)$ always has algebra generators of degree $1$. Furthermore, \cite[Proposition 5.5]{Thaithesis} shows that if $\R_s(I)$ is generated in degrees $1,a,b$ then $\svd(I)\le \lcm(a,b)$. Hence, an example for $\R_s(I)$ that satisfies the above inequality must be generated in at least $4$ different degrees. We shall provide an example where $\R_s(I)$ is generated in exactly $4$ degrees. \par

\vspace{0.5em}

Consider the edge ideal of an odd cycle of length $2n+1$:
$$I(C_{2n+1}) = (x_1x_2,x_2x_3,\ldots ,x_{2n}x_{2n+1},x_{2n+1}x_1) \subset R=\kk[x_1,\ldots,x_{2n+1}].$$ 
By \cite[Theorem 3.4]{GHOS20} we have $I^{(k)}=I^k$ for all $k \le n$ and $I^{(n+1)} \not = I^{n+1}$. 
Moreover, $\mathcal{R}_s(I)=R[It,x_1\ldots x_{2n+1}t^{n+1}]$. 
We obtain the following consequence of this result.

\begin{lemma}
    \label{lem.UniCyclic}
    Let $I=I(C_{2n+1})$. Then $\mathcal{R}_s(I)^{(d)}$ is standard graded if and only if $n+1$ divides $d$. In particular, $\svd(I)=n+1$.
\end{lemma}

\begin{proof}
    Let $\omega =x_1\ldots x_{2n+1}$. By \cite[Theorem 3.4]{GHOS20}, we have $\mathcal{R}_s(I)=R[It, \omega t^{n+1}]$, thus, $\mathcal{R}_s(I)^{(n+1)}$ is standard graded. 
    Therefore, if $d$ is divisible by $n+1$ then $\mathcal{R}_s(I)^{(d)}$ is also standard graded. 
    Conversely, suppose that $\mathcal{R}_s(I)^{(d)}$ is standard graded for some $d>0$. Consider the element $\omega^dt^{d(n+1)} \in \mathcal{R}_s(I)_{d(n+1)}=\left( \mathcal{R}_s(I)_{d}\right)^{n+1}$. 
    Thus, $\omega^d = m_1 m_2 \cdots m_{n+1}$ where $m_it^d \in \mathcal{R}_s(I)_{d}$ for all $i$. 
    Since $\mathcal{R}_s(I)$ is generated by $\omega t^{n+1}$ and $It$, for each $i$, we can write $m_it^d = \omega^{c_i}t^{c_i(n+1)} u_iE_it^{q_i}$ where $u_i \in R$ is a monomial and $E_i$ is a product of exactly $q_i$ minimal generators of $I$. 
    Let $c=c_1+\cdots + c_{n+1}$, $q=q_1+\cdots + q_{n+1}$, and $U=u_1\ldots u_{n+1}$. 
    From the equality $\omega^dt^{d(n+1)} = \prod_{i=1}^{n+1} \omega^{c_i}t^{c_i(n+1)} u_iE_it^{q_i}$, comparing $t$-degree yields $d(n+1)=c(n+1)+q$. This implies that $d\ge c$.
    Moreover, comparing the total degree in the $x$ variables yields: 
    \[
    d(2n+1)=c(2n+1)+\deg(U)+2q \implies \deg(U)=c-d.
    \]
    Since $\deg(U) \ge 0$, we have $c\ge d$. 
    Therefore, $c=d$ and $q=0$. 
    It follows that $q_i=0$ and $u_i=E_i=1$ for all $i$. 
    In particular, $d=c_i(n+1)$, hence, $d$ is divisible by $n+1$.
\end{proof}

We deduce the following result regarding the edge ideal of a disjoint union of odd cycles.

\begin{lemma}
\label{lem.DisjointCycles}
Let $I = I(C_{2n_1+1}) + I(C_{2n_2+1}) + \dots + I(C_{2n_m+1})$ be the edge ideal of the disjoint union of $m$ odd cycles. Let $L = \lcm(n_1+1, n_2+1, \dots, n_m+1)$. If $\mathcal{R}_s(I)^{(d)}$ is standard graded, then $d$ is a multiple of $L$.
\end{lemma}

\begin{proof}
Let $G = \bigsqcup_{i=1}^m C_{2n_i+1}$ be the given graph, $V$ be its vertex set, and $V_j$ be the vertex set of the $j$-th cycle $C_{2n_j+1}$. Let $R = \kk[x_v \mid v \in V]$ and $R_j = \kk[x_v \mid v \in V_j]$ be their respective polynomial rings. Denote $I_j=I(C_{2n_j+1})$ for each $j=1,\ldots ,m$.

Assume that $\mathcal{R}_s(I)^{(d)}$ is standard graded. Fix an arbitrary index $j \in \{1, \dots, m\}$. We wish to show that $n_j+1$ divides $d$. 

By \cite[Theorem 3.4]{HNTT20}, for each $n$, we have $(I_1 + I_2)^{(n)} = \sum_{k=0}^n I_1^{(k)} I_2^{(n-k)}$. 
This implies that $\mathcal{R}_s(I_1 + I_2) \cong \mathcal{R}_s(I_1) \otimes_\kk \mathcal{R}_s(I_2)$, 
therefore, 
$$\mathcal{R}_s(I) \cong \mathcal{R}_s(I_1) \otimes_\kk \mathcal{R}_s(I_2) \otimes_\kk \cdots  \otimes_\kk \mathcal{R}_s(I_m).$$

In particular, writing $\omega_i = \prod_{v\in V_i} x_v$ for each $i$, we have
$$\mathcal{R}_s(I) = R \left[I_1 t, \ldots , I_m t, \omega_1 t^{n_1+1}, \ldots, \omega_m t^{n_m+1} \right].$$

Define an algebra homomorphism $\pi_j: R \to R_j$ by setting $\pi_j(x_v) = x_v$ if $v \in V_j$, and $\pi_j(x_v) = 0$ if $v \notin V_j$. This naturally induces a graded algebra homomorphism on the symbolic Rees algebras mapping $t \mapsto t$:
$\Phi_j: \mathcal{R}_s(I) \to \mathcal{R}_s(I_j)$.

Furthermore, applying $\Phi_j$ to the generators of $\mathcal{R}_s(I)$ shows that $\Phi_j$ is a surjective graded homomorphism that restricts to a surjective graded homomorphism on the Veronese subalgebras
$\Phi_j: \mathcal{R}_s(I)^{(d)} \twoheadrightarrow \mathcal{R}_s(I_j)^{(d)}.$

Since the homomorphic image of a standard graded algebra is standard graded, we have $\mathcal{R}_s(I_j)^{(d)}$ is also standard graded. By Lemma \ref{lem.UniCyclic}, $n_j+1$ divides $d$. Since this is true for all $j$, $d$ must be divisible by $L$.
\end{proof}

We are now ready to provide a negative answer to \cite[Question 3.10]{symbolicReesSurvey}.

\begin{example}
    \label{ex.3DisjOddCycles} 
    Let $I_1=I(C_{11}) \subset \kk[x_1,\ldots ,x_{11}]$, $I_2=I(C_{19}) \subset \kk[y_1,\ldots ,y_{19}]$, and $I_3=I(C_{29}) \subset \kk[z_1,\ldots ,z_{29}]$ be the respective edge ideals of the three odd cycles in polynomial rings over disjoint sets of variables. Let $I=I_1+I_2+I_3$ be the edge ideal of the disjoint union of the three odd cycles.
    As in the proof of \Cref{lem.DisjointCycles}, we have
    $$\mathcal{R}_s(I) = R \left[I_1 t, I_2 t,  I_3 t,  w_1 t^6,  w_2 t^{10},  w_3 t^{15} \right] = R \left[ I t,  w_1 t^6, w_2 t^{10}, w_3 t^{15} \right],$$
    where $w_1 = \prod_{i=1}^{11} x_i$, $w_2 = \prod_{j=1}^{19} y_j$, $w_3 = \prod_{k=1}^{29} z_k$, and $R$ is the polynomial ring on the $x,y,z$'s variables. 
    Consider the element $f=(x_1x_2t)(w_1t^6)^4(w_2t^{10})^2(w_3t^{15}) \in \mathcal{R}_s(I)_{60}$. We claim that $f \notin (\mathcal{R}_s(I)_{30})^2$. 
    
    We can check that $(1,4,2,1;2)$ is a minimal solution of the equation $x+6y+10z+15t=30u$. Hence, by looking at $t$-degree of $f$, it suffices to show that $f$ has a unique factorization into the algebra generators of $\mathcal{R}_s(I)$.
    Suppose $f$ factors into the algebra generators. 
    Because $\mathcal{R}_s(I)_0 = R$, this factorization may include a monomial $r \in R$ of $t$-degree 0.
    Let $a_x, a_y,$ and $a_z$ denote the number of generators of the form $x_ix_j t, y_iy_j t,$ and $z_iz_j t$, respectively, (where $x_ix_j \in I_1, y_iy_j \in I_2, z_iz_j \in I_3$), and $b, c,$ and $d$ be the exponents of $w_1 t^6, w_2 t^{10},$ and $w_3 t^{15}$, respectively, in this factorization of $f$. 

    Comparing the total degree of each variable set yields the following inequalities:
    \[
        2a_z + 29d \le 29 , 2a_y + 19c \le 38, 2a_x + 11b \le 46.
    \]
    Comparing the $t$-degree yields:
    $$ (a_z + 15d) + (a_y + 10c) + (a_x + 6b) = 60. $$
    This implies that 
    \[
    120 \le (29+d) + (38+c) + (46+b) = 113+b+c+d \implies b+c+d \ge 7.
    \]
    On the other hand, since $2a_z + 29d \le 29$, we have $d \le 1$. Similarly, $c \le 2$ and $b \le 4$. Therefore, $d=1, a_z=0$, $c=2, a_y=0$, $b=4, a_x=1$, and $r=1$. This shows $f$ has a unique factorization into the algebra generators of $\mathcal{R}_s(I)$, therefore $f \notin (\mathcal{R}_s(I)_{30})^2$.
    
    Thus, $\mathcal{R}_s(I)^{(30)}$ is not standard graded. By \Cref{lem.DisjointCycles}, $\svd(I)>30 = \lcm(1,6,10,15)$. Note that by \Cref{lem.xi<n}, any minimal solution $(x,y,z,t;u)$ must satisfy 
    \[
    30u \le \left(\frac{30}{1}-1 \right) +6\left(\frac{30}{6}-1 \right) + 10\left(\frac{30}{10}-1 \right) +15\left(\frac{30}{15}-1 \right) = 88 \implies u\le 2.
    \]
    This implies that there is no minimal solutions with $u\ge 3$, therefore, $\mathcal{R}_s(I)^{(60)}$ is standard graded, hence by \Cref{lem.DisjointCycles}, $\svd(I)=60=2\lcm(1,6,10,15)$. One can also check that $(1,4,2,1;2)$ is the only minimal solution with $u=2$. 
\end{example}

\begin{remark}
    \label{rem.3DisjOddCycles}
    More generally, we believe that one can construct many such examples by letting $a=pq, b=qr, c=rp$ where $p,q,r$ are three distinct primes, then examining when the equation $x+ay+bz+ct=\lcm(a,b,c)u$ does not satisfy condition $\M$ and finding minimal solutions with $u \ge 2$. 
    Examining the first $5$ prime numbers, computations show that besides \Cref{ex.3DisjOddCycles}, the following cases also provide examples such that the equation has minimal solutions with $u=2$. 
    Using a similar argument as in \Cref{ex.3DisjOddCycles}, in the following cases, one can show that the minimal solutions with $u=2$ give rise to examples such that 
    $\svd(I) > \text{ the lcm of the degrees of a set of algebra generators for } \R_s(I)$.
    \begin{itemize}
        \item $(p,q,r)=(2,3,11)$: the minimal solution $(1,9,2,1;2)$ gives rise to the example $I=I(C_{11})+I(C_{43})+I(C_{65})$. 
        \item $(p,q,r)=(2,7,11)$: the minimal solutions $(1,7,6,1;2)$ and $(3,10,4,1;2)$ both give rise to the example $I=I(C_{27})+I(C_{43})+I(C_{153})$. 
        \item $(p,q,r)=(3,5,7)$: the minimal solutions $(1,6,4,1;2)$ and $(2,5,3,2;2)$ both give rise to the example $I=I(C_{29})+I(C_{41})+I(C_{69})$. 
        \item $(p,q,r)=(3,5,11)$: the minimal solution $(1,8,3,2;2)$ gives rise to the example $I=I(C_{29})+I(C_{65})+I(C_{109})$. 
    \end{itemize}
    Further computations show that in all above cases, no solutions with $u\ge 3$ are minimal. Note that by \Cref{lem.xi<n}, any minimal solution of the equation $x+ay+bz+ct=du$ where $d=\lcm(a,b,c)$ must satisfy $du \le 4d-1-a-b-c$, hence, must have $u \le 3$. One can check that there are no minimal solutions with $u=3$ and all minimal solutions with $u=2$ are shown above. In particular, for the above examples, $\svd(I)=2\lcm(1,a,b,c)$ where $1,a,b,c$ are degrees of algebra generators of $\R_s(I)$.
\end{remark}

It would be interesting to know if Rees's bound $\svd(I)\le s\lcm(a_1,\ldots ,a_s)$ when the symbolic Rees algebra of $I$ is generated in degrees $a_1,a_2,\ldots,a_s$ could be improved. It remains open and is interesting to know if $\svd(I)$ is an upper bound for the maximum degree of algebra generators for $\R_s(I)$ as asked in \cite[Question 3.11]{symbolicReesSurvey}.

\bibliographystyle{alpha}
\bibliography{References}

\end{document}